\documentclass[a4paper,10pt,leqno]{amsart}
\usepackage{amsmath}
\usepackage{amsfonts}
\usepackage{amssymb}
\usepackage[english]{babel}
\usepackage{graphicx}
\usepackage[normalem]{ulem}
\usepackage[makeroom]{cancel}
\usepackage{amsthm}
\usepackage{color}

\newtheorem{theorem}{\textbf{Theorem}}[section]
\newtheorem{lemma}[theorem]{\textbf{Lemma}}
\newtheorem{proposition}[theorem]{\textbf{Proposition}}

\newtheorem{rem}[theorem]{\textbf{Remark}}

\title{Blow-up for semilinear parabolic equations\\ in cones of the hyperbolic space}

\date{\today}

\keywords{}

\subjclass[2020]{}

\begin{document}
\maketitle

\begin{center}
\textsc{\textmd{D. D. Monticelli\footnote{Politecnico di Milano, Italy. Email:
dario.monticelli@polimi.it.} and F. Punzo\footnote{Politecnico di Milano, Italy. Email: fabio.punzo@polimi.it.}}}
\end{center}

\begin{abstract} We investigate existence and nonexistence of global in time nonnegative solutions to the semilinear heat equation, with a reaction term of the type $e^{\mu t}u^p\, (\mu\in\mathbb{R}, p>1),$ posed on cones of the hyperbolic space. Under a certain assumption on $\mu$ and $p$, related to the bottom of the spectrum of $-\Delta$ in $\mathbb H^n$, we prove that any solution blows up in finite time, for any nontrivial nonnegative initial datum. Instead, if the parameters $\mu$ and $p$ satisfy the opposite condition we have: (a) blow-up when the initial datum is large enough, (b) existence of global solutions when the initial datum is small enough. Hence our conditions on the parameters $\mu$ and $p$ are optimal.

We see that blow-up and global existence do not depend on the amplitude of the cone. This is very different from what happens in the Euclidean setting (see \cite{BaLe}), and it is essentially due to a specific geometric feature of $\mathbb H^n$.
\end{abstract}

\maketitle


\section{Introduction}
Let $\mathbb H^n$ be the $n-$dimensional hyperbolic space, and let $r=r(x)$ denote the distance of a point $x\in\mathbb{H}^n$ from a fixed reference point $o$. Consider the unit sphere $\mathcal S^{n-1}\subset\mathbb{H}^n$ centered at $o$ 
and let $\Omega\subseteq\mathcal{S}^{n-1}$ be a relatively open domain. Let $\mathcal{C}$ be the cone with vertex at $o\in\mathbb{H}^n$ determined by the domain $\Omega$.
We introduce global polar coordinates $(r,\theta)\in [0,\infty)\times\mathbb{S}^{n-1}$ centered at $o$ in $\mathbb{H}^n$, where $\mathbb{S}^{n-1}$ denotes the unit sphere in $\mathbb{R}^n$. With a slight abuse of notation we will identify $\mathcal{S}^{n-1}\equiv\mathbb{S}^{n-1}$. Then the cone $\mathcal{C}$ can be described as follows
$$
\mathcal{C}=\{x=(r,\theta)\in\mathbb{H}^n\,|\,r>0,\,\theta\in\Omega\}=(0,\infty)\times\Omega.
$$

In this paper, we investigate finite-time blow-up of nonnegative solutions to problem
\begin{equation}\label{7}
\begin{cases}
u_t-\Delta u=e^{\mu t}u^p&\text{ in }\mathcal{C}\times(0,T),\\
u=0&\text{ in }\partial\mathcal{C}\times(0,T),\\
u=u_0&\text{ on }\mathcal{C}\times\{0\},
\end{cases}
\end{equation}
with $$p>1,\,\,\,\,\,\, T>0,\,\,\,\,\,\, \mu\in\mathbb{R},\,\,\,\,\,\, u_0\geq0\text{ continuous and bounded}.$$ Here and in the following $T\in(0,\infty]$ denotes the maximal existence time of the solution $u$, while $\Delta$ stands for the Laplace-Beltrami operator on $\mathbb H^n$. It is clear that, when we consider the case $\Omega=\mathbb S^{n-1}$, we have $\mathcal C=\mathbb H^n$ and no boundary conditions are given. Hence problem \eqref{7} in this case reads
\begin{equation}\label{7c}
\begin{cases}
u_t-\Delta u=e^{\mu t}u^p&\text{ in }\mathbb H^n\times(0,T),\\
u=u_0&\text{ on }\mathbb H^n\times\{0\}.
\end{cases}
\end{equation}

\smallskip

In \cite{BaLe}, the authors addressed problem
\begin{equation}\label{7b}
\begin{cases}
u_t-\Delta u=u^p&\text{ in }\mathfrak{D}\times(0,T),\\
u=0&\text{ in }\partial\mathfrak{D}\times(0,T),\\
u=u_0&\text{ on }\mathfrak{D}\times\{0\},
\end{cases}
\end{equation}
where $$\mathfrak D:=\{x=(r, \theta)\in \mathbb R^n: r>0, \theta\in \Omega\}$$
is a cone of $\mathbb R^n$, $(r, \theta)\in [0, +\infty)\times \mathbb S^{n-1}$ are polar coordinates in $\mathbb R^n$ and $\Omega$ is a relatively open domain in the unit sphere $\mathbb{S}^{n-1}\subset\mathbb{R}^n$. It is shown that existence and nonexistence of global solutions is closely related to the domain $\Omega$\,. To be more specific,
we introduce some preliminary material. Let $\Delta_\theta$ be the Laplace-Beltrami operator on $\mathbb{S}^{n-1}$, and consider the eigenvalue problem on $\Omega$
\begin{equation}\label{8}
\begin{cases}
-\Delta_\theta \psi=\omega\psi&\text{ in }\Omega,\\
\psi=0&\text{ on }\partial\Omega.
\end{cases}
\end{equation}
In particular we denote by $\omega_1\geq0$ the first eigenvalue of problem \eqref{8} and by $\psi_1$ the corresponding positive eigenfunction, normalized so that
$$
\|\psi_1\|_{L^1(\Omega)}=\int_\Omega \psi_1 d\sigma_\theta=1.
$$
Obviously $\omega_1=0$ if and only if $\Omega=\mathbb{S}^{n-1}$, that is $\mathcal C=\mathbb H^n$ and $\mathfrak D=\mathbb R^n$. In this case $\psi_1$ is a positive constant.

Let $\lambda:=-\gamma_-$, where $\gamma_-$ is the negative root of
\[\gamma(\gamma+n-2)=\omega_1\,.\]
In \cite{BaLe} it is shown that if $$1<p<1+\frac2{2+\lambda},$$ then any solution to problem \eqref{7b} blows up in finite time. On the other hand, problem \eqref{7b} possesses global in time solutions, for suitable initial data $u_0$, provided that
\[1+\frac2{\lambda}<p<\frac{n+1}{n-3} \quad \text{ for } n>3,\]
\[p>1+\frac2{\lambda} \quad \text{ for } n=2, 3\,.\]
Note that conditions on $p$ ensuring existence or nonexistence of global in time solutions of \eqref{7b} depend on $\Omega$ through the quantity $\lambda$.

Let us mention that when $\mathfrak D=\mathbb R^n$, problem \eqref{7b} was studied in \cite{Fuj} (see also \cite{Hay}). It is shown that any solution blows up in finite time, whenever
\[1<p\leq 1+\frac 2n,\]
while a global in time solution exists, provided that
\[p>1+\frac 2n\,,\]
and the initial datum is small enough.

In \cite{BaPoTe} the authors studied problem \eqref{7c}. We recall that $\lambda_1(\mathbb{H}^n)=\frac{(n-1)^2}{4}$ is the bottom of the spectrum of the operator $-\Delta$ on $\mathbb{H}^n$.
It is then established that if
\begin{equation}\label{13}
  \mu>(p-1)\lambda_1(\mathbb{H}^n),
  \end{equation}
then any solution blows up in finite time. On the other hand (see also \cite{WaYi}), if
\begin{equation}\label{10}
  \mu\leq(p-1)\lambda_1(\mathbb{H}^n),
\end{equation}
and $\mu\geq0$ then a global in time solution exists, provided that $u_0$ is sufficiently small.

\smallskip

In this paper we show that existence and nonexistence of global solutions to problem \eqref{7} on $\mathcal{C}\subseteq\mathbb{H}^n$ do not depend on the domain $\Omega\subseteq\mathbb{S}^{n-1}$, and that the critical exponent phenomenon is always the same as for the whole space $\mathbb{H}^n$. This is in striking contrast with respect to the Euclidean case. In fact, consider any $\Omega\subseteq \mathbb S^{n-1}$. We prove that:
\begin{itemize}
\item[i.] if condition \eqref{13} holds, then any solution to problem \eqref{7} blows up in finite time;
\item[ii.] if condition \eqref{10} holds and $u_0$ is large enough, then the corresponding solution to problem \eqref{7} blows up in finite time;
\item[iii.] if condition \eqref{10} holds and $u_0$ is small enough, then there exists a global in time solution to problem \eqref{7}.
\end{itemize}
Observe that the conditions which guarantee existence or nonexistence of global in time solutions are completely independent of $\Omega$. Moreover, since condition \eqref{10} is the opposite of \eqref{13}, we see that our assumptions on the parameters $\mu$ and $p$ are optimal. We also explicitly remark that, for every cone $\mathcal{C}\subseteq\mathbb{H}^n$, the threshold case $\mu=(p-1)\lambda_1(\mathbb{H}^n)$ belongs to the regime where global in time solutions can exist; on the contrary, in the case of $\mathbb{R}^n$, the threshold case $p=1+\frac{2}{n}$ belongs to the regime where all solutions exhibit blow up in finite time.

Note that when $\Omega=\mathbb S^{n-1}$, and thus problem \eqref{7} becomes \eqref{7c}, our results recover those established in \cite{BaPoTe, WaYi} for problem \eqref{7c}.

Let us also explicitly note that when $\mu=0$, condition \eqref{10} is always fulfilled. Hence from our results it follows that, for every $p>1$, any solution to problem
\begin{equation}\label{7d}
\begin{cases}
u_t-\Delta u=u^p&\text{ in }\mathcal{C}\times(0,T),\\
u=0&\text{ in }\partial\mathcal{C}\times(0,T),\\
u=u_0&\text{ on }\mathcal{C}\times\{0\}
\end{cases}
\end{equation}
blows up in finite time, provided that $u_0$ is large enough. On the other hand, for every $p>1$, problem \eqref{7d} admits a global in time solution, whenever $u_0$ is sufficiently small.
\smallskip

In order to briefly explain such enormous difference between $\mathbb R^n$ and $\mathbb H^n$, consider the Laplace operator in polar coordinates.
In $\mathbb R^n$, obviously we have:
$$
\Delta u=u_{rr}+\frac{n-1}r \, u_r+\frac{1}{r^2}\Delta_\theta u,
$$
where the subscript $r$ is used to denote a partial derivative with respect to that variable. On the other hand, on $\mathbb H^n$ we have
$$
\Delta u=u_{rr}+(n-1)\coth r \, u_r+\frac{1}{(\sinh r)^2}\Delta_\theta u.
$$
Observe that in $\mathbb R^n$ the coefficient $\frac 1{r^2}$ multiplies $\Delta_{\theta}$, whereas in $\mathbb H^n$ we have $\frac{1}{(\sinh r)^2}$. Clearly, $\frac{1}{(\sinh r)^2}$ tends to $0$ as $r$ tends to $+\infty$ much faster than
$\frac 1{r^2}$. Consequently, in $\mathbb R^n,$ the term $\frac{1}{r^2}\Delta_\theta u$ plays a certain role and existence and nonexistence results depend on the domain $\Omega$ where $\Delta_\theta$ is considered. On the other hand, in $\mathbb H^n$ the term $\frac{1}{(\sinh r)^2}\Delta_\theta u$ is somehow negligible. So, $\Omega$ does not have any role in determining existence or nonexistence of global solutions.

The proof of finite-time blow-up is based on a generalization of the Kaplan method (see \cite{Kapl}; see also \cite{BaLe}) on unbounded cones of $\mathbb H^n$, with differential equations involving time dependent coefficients. In particular we show that, if $u$ is a solution of problem \eqref{7}, for functions $\Phi$ satisfying suitable conditions, the function
\[ G(t)=\int_0^\infty\int_\Omega e^{\alpha t}u(r,\theta,t)\psi_1(\theta)\Phi(r)(\sinh r)^{n-1}\,d\sigma_\theta dr\]
is a supersolution of an appropriate ordinary differential equation. In such equation a time dependent coefficient appears. To obtain this we develop some ideas used in \cite{BaLe}. Then we specialize our choice of the function $\Phi$, which is a subsolution of an appropriate auxiliary problem. This point is very different from the Euclidean case addressed in \cite{BaLe}, since it is strictly related to the geometry of $\mathbb H^n$. Moreover, in \cite{BaLe} the authors did not need to consider time dependent coefficients, as instead we do here.

In view of our choice of $\Phi$, we can infer that $G(t)$ is a supersolution of another suitable ordinary differential equation. From this we can deduce that $G(t)$
must blow up in finite time, and this yields blow-up for solutions to problem \eqref{7}.

Finally, global existence of solutions for small nonnegative initial data $u_0$ is obtained by means of suitable supersolutions and local barrier functions at points of the boundary of the domain.

\medskip

We conclude the introduction with a final remark. In \cite{BaPoTe} the authors showed also that if $q>-1$, $p>1$ and $u_0\geq0$ on $\mathbb{H}^n$ is sufficiently small, then there exist global solutions of problem
$$
\begin{cases}
  u_t-\Delta u=t^qu^p & \mbox{in } \mathbb{H}^n\times(0,T), \\
  u=u_0 & \mbox{on }\mathbb{H}^n\times\{0\}.
\end{cases}
$$
Adapting the arguments used to study problem \eqref{7} one can show that if $q>-1$, $p>1$, $\Omega\subseteq\mathbb{S}^{n-1}$ is a relatively open domain and $\mathcal{C}=(0,\infty)\times\Omega\subseteq\mathbb{H}^n$, then
\begin{equation}\label{7e}
\begin{cases}
u_t-\Delta u=t^q u^p&\text{ in }\mathcal{C}\times(0,T),\\
u=0&\text{ in }\partial\mathcal{C}\times(0,T),\\
u=u_0&\text{ on }\mathcal{C}\times\{0\},
\end{cases}
\end{equation}
admits global in time solutions, provided $u_0\geq0$ on $\mathcal{C}$ is sufficiently small; on the other hand, if $u_0\geq0$ on $\mathcal{C}$ is large enough, then the corresponding solution of \eqref{7e} blows up in finite time.

\medskip

The paper is organized as follows. In Section \ref{mr} we state our main results concerning both finite-time blow-up of solutions and global existence. In Section \ref{prel} we deduce the key property of the function $G$. Then in Section \ref{nonexi} we introduce the function $\Phi$ and we show our blow-up results, while in Section \ref{exi} we prove existence of global in time solutions for suitably small nonnegative initial data.

\section{Main results}\label{mr}
We deal with solutions $u\in C^{2,1}_{x,t}(\mathcal{C}\times(0,T))$ of \eqref{7}, which achieve the initial datum $u_0$ continuously. Moreover,
\[u(t)\in L^\infty(\mathcal C)\quad \text{ for any }\,\, t\in (0, T).\]
If $T<+\infty$ and
\[\|u(t)\|_{L^\infty(\mathcal C)}\to +\infty \quad \text{ as }\,\, t\to T^-,\]
we say that $u$ blows up in finite time, whereas
if $T=+\infty$, we say that $u$ is a global in time solution.

Concerning finite-time blow-up of solutions to \eqref{7}, for any $u_0\geq0$, $u_0\not \equiv 0$, whenever \eqref{13} if fulfilled, we prove the following
\begin{theorem}\label{thm1}
  Let $u$ be a solution of \eqref{7} with $\Omega\subseteq\mathbb{S}^{n-1}$, $\mathcal{C}=(0,\infty)\times\Omega$, $p>1$, $u_0\geq0$, $u_0\not\equiv0$ continuous and bounded. Assume \eqref{13}.
  Then $u$ blows up in finite time. More precisely one has
  $$
  \|u(t)\|_{L^\infty(\mathcal{C})}\rightarrow+\infty\qquad\text{ as }t\rightarrow T^-
  $$
  with $0<T<\infty$.
\end{theorem}
\begin{rem}
  As it is clear from the proof, we have $$T\leq\frac{(G(0))^{1-p}}{p-1},$$ with $G(0)>0$ depending on the initial condition $u_0$ and given by formula \eqref{12}.
\end{rem}

In addition, in the complementary range of values of $\mu$ considered in Theorem \ref{thm1} we again have finite-time blow-up of solutions to \eqref{7}, provided that the initial datum is large enough. This is the content of the following
\begin{theorem}\label{thm2}
  Let $u$ be a solution of \eqref{7} with $\Omega\subseteq\mathbb{S}^{n-1}$, $\mathcal{C}=(0,\infty)\times\Omega$, $p>1$. Suppose that $u_0\geq0$, $u_0\not\equiv0$ is continuous, bounded and large enough, in a suitably weighted $L^1$ sense. Assume \eqref{10}.
  Then $u$ blows up in finite time.
\end{theorem}
\begin{rem}
   More precisely we prove the following: assume \eqref{10}, let $\alpha>\lambda_1(\mathbb{H}^n)$, $m\geq2$, $m(m-1)>\omega_1$, $k\in(0,k_0]$ with $k_0$ given by Lemma \ref{lemma1} and assume
   \begin{align}\label{14}
   &\int_0^{\infty}\int_\Omega u_0(r,\theta)\psi_1(\theta)r^me^{-kr^2}(\sinh r)^{n-1}\,d\sigma_\theta dr\\
   \nonumber&\qquad>\left(\alpha-\frac{\mu}{p-1}\right)^{\frac{1}{p-1}}\int_0^{\infty}r^me^{-kr^2}(\sinh r)^{n-1}\,dr,
   \end{align}
then
  $$
  \|u(t)\|_{L^\infty(\mathcal{C})}\rightarrow+\infty\qquad\text{ as }t\rightarrow T^-
  $$
  with $0<T\leq T^*<+\infty$ and $T^*$ given in \eqref{20}.
\end{rem}

Finally, we consider the same range of values of $\mu$ treated in Theorem \ref{thm2}, but instead of supposing that $u_0$ is nonnegative and large enough, on the contrary now we assume that the initial datum is nonnegative and sufficiently small. In this case, we get global in time existence of solutions to \eqref{7}. In fact, we show the following

\begin{theorem}\label{thm3}
  Let $\Omega\subseteq\mathbb{S}^{n-1}$, $\mathcal{C}=(0,\infty)\times\Omega$, $p>1$. Assume \eqref{10}.  Suppose that $u_0$ is continuous, bounded and small enough.
  Then there exists a global in time solution $u$ of problem \eqref{7}. In addition $u\in L^\infty(\mathcal{C}\times(0,\infty))$.
\end{theorem}

\begin{rem}
The hypothesis $u_0$ sufficiently small made in Theorem \ref{thm3} can be formulated more precisely. Indeed, as we will see in  Lemma \ref{lemma2} below, if \eqref{10} holds there exists a
 positive supersolution $v$ of
  \begin{equation}\label{15}
  v_t-\Delta v=e^{\mu t}v^p\qquad\text{ in }\mathbb{H}^n\times(0,\infty)\,.
  \end{equation}
 We prove Theorem \ref{thm3} under the assumption that
  $$0\leq u_0(x)\leq v(x,0)\qquad\text{ for every }x\in\mathcal{C}.$$
\end{rem}

Concerning problem \eqref{7e}, we have that solutions with large enough nonnegative initial data $u_0$ blow up in finite time, while suitably small nonnegative initial data give rise to global in time bounded solutions. We state these results in the following theorems.

\begin{theorem}\label{thm2bis}
  Let $u$ be a solution of \eqref{7e} with $\Omega\subseteq\mathbb{S}^{n-1}$, $\mathcal{C}=(0,\infty)\times\Omega$, $p>1$, $q>-1$. Suppose that $u_0\geq0$, $u_0\not\equiv0$ is continuous, bounded and large enough, in a suitably weighted $L^1$ sense. Then $u$ blows up in finite time.
\end{theorem}

\begin{theorem}\label{thm3bis}
  Let $\Omega\subseteq\mathbb{S}^{n-1}$, $\mathcal{C}=(0,\infty)\times\Omega$, $p>1$, $q>-1$. Suppose that $u_0$ is continuous, bounded and small enough.
  Then there exists a global in time solution $u$ of problem \eqref{7e}. In addition $u\in L^\infty(\mathcal{C}\times(0,\infty))$.
\end{theorem}

\section{A preliminary estimate}\label{prel}
The following proposition can be regarded as a generalization of the Kaplan method in unbounded cones of $\mathbb H^n$. Moreover, a time dependent coefficient appears in the estimate.

\begin{proposition}\label{prop1}
  Let $u$ be a solution of problem \eqref{7}, $\alpha\in\mathbb{R}$ and $\Phi\in C^2([0,\infty))$ be a nonnegative function satisfying the following assumptions
  \begin{itemize}
    \item[i)] $\Phi(r)\sim l\, r^m$ as $r$ tends to $0$, with $l\in(0,\infty)$ and $m\geq2$;
    \item[ii)] $\displaystyle\int_0^\infty\Phi(r)(\sinh r)^{n-1}\,dr=1$;
    \item[iii)] $\displaystyle\int_0^\infty\left(|\Phi_r(r)|+|\Phi_{rr}(r)|\right)(\sinh r)^{n-1}\,dr<\infty$.
  \end{itemize}
  Then the function
\begin{equation}\label{eq1a}
  G(t)=\int_0^\infty\int_\Omega e^{\alpha t}u(r,\theta,t)\psi_1(\theta)\Phi(r)(\sinh r)^{n-1}\,d\sigma_\theta dr
\end{equation}
satisfies
\begin{align}\label{eq2a}
&G'(t)\geq e^{(\mu-(p-1)\alpha)t}(G(t))^p\\
\nonumber&\,\,\,\,+\int_0^\infty\int_\Omega e^{\alpha t}u(r,\theta,t)\Psi_1(\theta)\left(\Delta\Phi(r)+\alpha\Phi(r)-\frac{\omega_1}{(\sinh r)^2}\Phi(r)\right)(\sinh r)^{n-1}\,d\sigma_\theta dr
\end{align}
for all $t\in(0,T)$.
\end{proposition}

\begin{proof}
We consider the first eigenvalue $\omega_1\geq0$ and the first positive eigenfunction $\psi_1$ of problem \eqref{8}, normalized with $\|\psi_1\|_{L^1(\Omega)}=1$. We define
$$
\tilde{u}(r,t)=e^{\alpha t}\int_\Omega \psi_1(\theta)u(r,\theta,t)\,d\sigma_\theta
$$
for $r>0$, $t\in(0,T)$. Note that $\tilde{u}$ inherits the same regularity of $u$ in $(0,\infty)\times(0,T)$ and that
$$
G(t)=\int_0^\infty\tilde{u}(r,t)\Phi(r)(\sinh r)^{n-1}\,dr.
$$
Then we have
\begin{align*}
  \tilde{u}_t & =\alpha e^{\alpha t}\int_\Omega \psi_1u\,d\sigma_\theta+e^{\alpha t}\int_\Omega \psi_1u_t\,d\sigma_\theta \\
   & =\alpha \tilde{u} + e^{\alpha t}\int_\Omega \psi_1\Delta u\,d\sigma_\theta+e^{\alpha t}\int_\Omega e^{\mu t}\psi_1u^p\,d\sigma_\theta \\
   & =\alpha \tilde{u} + e^{\alpha t}\int_\Omega \psi_1\left(u_{rr}+(n-1)\coth r \, u_r+\frac{1}{(\sinh r)^2}\Delta_\theta u\right) \,d\sigma_\theta\\
   &\,\,\,\,\,\,+e^{(\alpha+\mu) t}\int_\Omega \psi_1u^p\,d\sigma_\theta\\
   &=\alpha\tilde{u}+\tilde{u}_{rr}+(n-1)\coth r \,\tilde{u}_r-\frac{\omega_1}{(\sinh r)^2}\tilde{u}+e^{(\alpha+\mu)t}\int_\Omega u^p\psi_1\,d\sigma_\theta,
\end{align*}
where we have integrated by parts twice on $\Omega$ in order to obtain the last equality. Then using Jensen's inequality we obtain
\begin{align*}
  \tilde{u}_t &=\alpha\tilde{u}+\Delta\tilde{u}-\frac{\omega_1}{(\sinh r)^2}\tilde{u}+e^{(\alpha+\mu)t}\int_\Omega u^p\psi_1\,d\sigma_\theta\\
  &\geq \alpha\tilde{u}+\Delta\tilde{u}-\frac{\omega_1}{(\sinh r)^2}\tilde{u}+e^{(\alpha+\mu)t}\left(\int_\Omega u\psi_1\,d\sigma_\theta\right)^p \\
  & =\alpha\tilde{u}+\Delta\tilde{u}-\frac{\omega_1}{(\sinh r)^2}\tilde{u}+e^{\left(\mu-(p-1)\alpha\right)t}\tilde{u}^p.
\end{align*}
Now for arbitrary $0<\varepsilon<R$ we consider a smooth cutoff function $\zeta=\zeta(r)$ such that $\zeta\equiv0$ in $(0,\tfrac{\varepsilon}{2})\cup(2R,\infty)$, $\zeta\equiv1$ in $(\varepsilon,R)$ and $0\leq\zeta\leq1$ in $(0,\infty)$, with
$$0\leq\zeta'\leq\frac{C}{\varepsilon}\quad\text{in }\left(\frac{\epsilon}{2},\epsilon\right),\qquad-\frac{C}{R}\leq\zeta'\leq0\quad\text{in }(R,2R),$$
and
$$|\zeta''|\leq\frac{C}{\varepsilon^2}\quad\text{in }\left(\frac{\epsilon}{2},\epsilon\right),\qquad |\zeta''|\leq\frac{C}{R^2}\quad\text{in }(R,2R),$$
for some $C>0$. Note that $$|\Delta\zeta|\leq\frac{C}{\varepsilon^2}\,\text{ on }\,\left(\frac{\epsilon}{2},\epsilon\right)\quad\text{ and }\quad|\Delta\zeta|\leq\frac{C}{R}\,\text{ on }\,(R,2R).$$
Define for every $t\in(0,T)$
\begin{align*}
G^{\varepsilon,R}(t)&=\int_0^\infty\int_\Omega e^{\alpha t}\zeta(r)u(r,\theta,t)\psi_1(\theta)\Phi(r)(\sinh r)^{n-1}\,d\sigma_\theta dr\\
&=\int_{\tfrac{\varepsilon}{2}}^{2R}\tilde{u}(r,t)\zeta(r)\Phi(r)(\sinh r)^{n-1}\,dr.
\end{align*}
Note that $t\mapsto G^{\varepsilon,R}(t)$ and $t\mapsto G(t)$ are functions of class $C^1$ in $(0,T)$. Using the above inequalities we obtain
\begin{align*}
(G^{\varepsilon,R})'(t)&\geq\int_0^\infty\left(\alpha\tilde{u}+\Delta\tilde{u}-\frac{\omega_1}{(\sinh r)^2}\tilde{u}\right)\zeta\Phi(\sinh r)^{n-1}\,dr\\
&\quad+e^{\left(\mu-(p-1)\alpha\right)t}\int_0^\infty\tilde{u}^p\zeta\Phi(\sinh r)^{n-1}\,dr.
\end{align*}
Note that integrating by parts we have
\begin{align*}
\int_0^\infty\Delta\tilde{u}\,\zeta\Phi(\sinh r)^{n-1}\,dr&=\int_0^\infty\tilde{u}\Delta(\zeta\Phi)(\sinh r)^{n-1}\,dr\\
&=\int_0^\infty\tilde{u}\left(\zeta\Delta\Phi+2\Phi_r\zeta_r+\Phi\Delta\zeta\right)(\sinh r)^{n-1}\,dr
\end{align*}
and hence
\begin{align*}
&(G^{\varepsilon,R})'(t)\geq\int_0^\infty\left(\alpha\Phi+\Delta\Phi-\frac{\omega_1}{(\sinh r)^2}\Phi\right)\tilde{u}\zeta(\sinh r)^{n-1}\,dr\\
&\,\,\,\,\,\,\,\,\,\,+e^{\left(\mu-(p-1)\alpha\right)t}\int_0^\infty\tilde{u}^p\zeta\Phi(\sinh r)^{n-1}\,dr+\int_0^\infty\tilde{u}\left(2\Phi_r\zeta_r+\Phi\Delta\zeta\right)(\sinh r)^{n-1}\,dr.
\end{align*}
Now we need to pass to the limit as $R$ tends to $\infty$ and $\varepsilon$ tends to $0$. Since $u(\cdot,t)$, and hence $\tilde{u}(\cdot,t)$, is bounded in $(0,\infty)$ for every fixed $t\in(0,T)$, by the integrability assumptions on $\Phi$ and its derivatives and by the asymptotic behavior of $\Phi$ close to $0$, we see that
$$
\int_0^\infty\left(\alpha\Phi+\Delta\Phi-\frac{\omega_1}{(\sinh r)^2}\Phi\right)\tilde{u}\zeta(\sinh r)^{n-1}\,dr+e^{\left(\mu-(p-1)\alpha\right)t}\int_0^\infty\tilde{u}^p\zeta\Phi(\sinh r)^{n-1}\,dr
$$
converges by the dominated convergence theorem to
$$
\int_0^\infty\left(\alpha\Phi+\Delta\Phi-\frac{\omega_1}{(\sinh r)^2}\Phi\right)\tilde{u}(\sinh r)^{n-1}\,dr+e^{\left(\mu-(p-1)\alpha\right)t}\int_0^\infty\tilde{u}^p\Phi(\sinh r)^{n-1}\,dr
$$
for every $t\in(0,T)$. Moreover we have
\begin{align*}
&\left|\int_0^\infty\tilde{u}\left(2\Phi_r\zeta_r+\Phi\Delta\zeta\right)(\sinh r)^{n-1}\,dr\right|\\
&\quad\leq\frac{C}{\varepsilon}\int_{\frac{\varepsilon}{2}}^\varepsilon|\Phi_r|(\sinh r)^{n-1}+\frac{C}{\varepsilon^2}\int_{\frac{\varepsilon}{2}}^\varepsilon\Phi(\sinh r)^{n-1}\\
&\quad\quad+\frac{C}{R}\int_{R}^{2R}|\Phi_r|(\sinh r)^{n-1}+\frac{C}{R}\int_{R}^{2R}\Phi(\sinh r)^{n-1}
\end{align*}
and thus the first integral converges to $0$. Now recall that $u(\cdot,t)\in L^\infty(\mathcal{C})$ for any $t\in (0, T)$. Moreover, note that $u_t(\cdot,t)\in L^\infty(\mathcal{C})$ for any $t\in (0, T)$: this can be obtained arguing as in the proof of \cite[Propositions 7.1.10 and 7.3.1]{Lun}. In view of $(ii)$ then we have $\Phi u(\cdot,t),\Phi u_t(\cdot,t)\in L^1(\mathcal{C})$ for every $t\in(0,T)$. Thus by the dominated convergence theorem we have that
\begin{align*}
(G^{\varepsilon,R})'(t)&=\int_0^\infty\int_\Omega (e^{\alpha t}u(r,\theta,t))_t\zeta(r)\psi_1(\theta)\Phi(r)(\sinh r)^{n-1}\,d\sigma_\theta dr\\
&\rightarrow \int_0^\infty\int_\Omega (e^{\alpha t}u_t(r,\theta,t))_t\psi_1(\theta)\Phi(r)(\sinh r)^{n-1}\,d\sigma_\theta dr=G'(t)
\end{align*}
for every $t\in(0,T)$. Then we obtain
\begin{align}
\label{9}G'(t)&\geq\int_0^\infty\left(\alpha\Phi+\Delta\Phi-\frac{\omega_1}{(\sinh r)^2}\Phi\right)\tilde{u}(\sinh r)^{n-1}\,dr\\
\nonumber&\,\,\,\,\,\,+e^{\left(\mu-(p-1)\alpha\right)t}\int_0^\infty\tilde{u}^p\Phi(\sinh r)^{n-1}\,dr
\end{align}
for every $t\in(0,T)$. In view of $(ii)$, by Jensen's inequality
$$
\int_0^\infty\tilde{u}^p\Phi(\sinh r)^{n-1}\,dr\geq\left(\int_0^\infty\tilde{u}\Phi(\sinh r)^{n-1}\,dr\right)^p=(G(t))^p
$$
for every $t\in(0,T)$. This, combined with \eqref{9} and recalling the definition of $\tilde{u}$, yields the conclusion.
\end{proof}

With the same arguments as those appearing in the proof of Proposition \ref{prop1}, where solutions of problem \eqref{7} are considered, one can show the following result, which involves solutions of problem \eqref{7e}.

\begin{proposition}\label{prop1bis}
  Let $u$ be a solution of problem \eqref{7e}, $\alpha\in\mathbb{R}$ and $\Phi\in C^2([0,\infty))$ be a nonnegative function satisfying assumptions (i), (ii), (iii) in Proposition \ref{prop1}. Then the function $G(t)$ defined for $t\in[0,T)$ by \eqref{eq1a} satisfies
\begin{align}\label{eq2abis}
&G'(t)\geq t^qe^{-(p-1)\alpha t}(G(t))^p\\
\nonumber&\,\,\,\,+\int_0^\infty\int_\Omega e^{\alpha t}u(r,\theta,t)\Psi_1(\theta)\left(\Delta\Phi(r)+\alpha\Phi(r)-\frac{\omega_1}{(\sinh r)^2}\Phi(r)\right)(\sinh r)^{n-1}\,d\sigma_\theta dr
\end{align}
for all $t\in(0,T)$.
\end{proposition}

\section{Blow-up: proofs}\label{nonexi}
Let $\Phi_0$ be defined on the hyperbolic space by
$$
\Phi_0(x)\equiv\Phi_0(r(x)):=r^me^{-kr^2}
$$
for $m\geq2$, $k>0$.

\begin{lemma}\label{lemma1}
   Consider a domain $\Omega\subseteq\mathbb{S}^{n-1}$ and let $\omega_1\geq0$ be the smallest Dirichlet eigenvalue of the Laplace-Beltrami operator $\Delta_\theta$ on $\Omega$. Let $$\alpha>\lambda_1(\mathbb{H}^n)$$ and $$m\geq2,\qquad\quad m(m-1)>\omega_1.$$ Then there exists $k_0=k_0(n,m,\alpha,\omega_1)>0$ such that for every $k\in(0,k_0]$ one has
  \begin{equation}\label{6}
  \Delta\Phi_0+\alpha\Phi_0\geq\frac{\omega_1}{(\sinh r)^2}\Phi_0\qquad\text{ in }\mathbb{H}^n.
  \end{equation}
\end{lemma}
\begin{proof}
  Since $m\geq2$ we have $\Phi_0\in C^2(\mathbb{H}^n)$ and
  \begin{align*}
    (\Phi_0)_r & =(mr^{m-1}-2kr^{m+1})e^{-kr^2} \\
    (\Phi_0)_{rr} & =\left(m(m-1)r^{m-2}-2k(2m+1)r^m+4k^2r^{m+2}\right)e^{-kr^2}
  \end{align*}
  Using polar coordinates $(r,\theta)\in[0,+\infty)\times\mathbb{S}^{n-1}$, a simple computation yields
  \begin{align*}
    & \Delta\Phi_0+\alpha\Phi_0-\frac{\omega_1}{(\sinh r)^2}\Phi_0 \\
    & =(\Phi_0)_{rr}+(n-1)\coth r (\Phi_0)_r+\alpha\Phi_0-\frac{\omega_1}{(\sinh r)^2}\Phi_0\\
 &=r^{m-2}e^{-kr^2}\Big[m(m-1)-2k(2m+1)r^2+4k^2r^4+m(n-1)r\coth r\\
 &\hspace{4,5cm}-2(n-1)kr^3\coth r+\alpha r^2-\frac{\omega_1}{(\sinh r)^2}r^2\Big]
\end{align*}
and the above quantity is nonnegative if and only if
\begin{align}
\label{5}&m(m-1)+m(n-1)r\coth r-\frac{\omega_1}{(\sinh r)^2}r^2\\
\nonumber&\quad+r^2\left[4k^2r^2-2(n-1)kr\coth r-(2k(2m+1)-\alpha)\right]\geq0.
\end{align}
We claim that there exists $R_0=R_0(\alpha)>1$ and $k_0=k_0(n,m,\alpha,\omega_1)>0$ such that
\begin{equation}\label{1}
4k^2r^2-2(n-1)kr\coth r-(2k(2m+1)-\alpha)\geq0
\end{equation}
for every $r\geq R_0$ and $0<k\leq k_0$. Indeed, since $\alpha>\lambda_1(\mathbb{H}^n)=\frac{(n-1)^2}{4}$, we can find $R_0>1$ such that
\begin{equation}\label{2}
\frac{(n-1)^2}{4}(\coth r)^2-\alpha<\frac{1}{2}\left(\frac{(n-1)^2}{4}-\alpha\right)<0
\end{equation}
for every $r\geq R_0$. Now let $k_0=k_0(n,m,\alpha,\omega_1)>0$ be such that for every $0<k\leq k_0$ one has
\begin{align}
\label{3}  & 2k(2m+1)+\frac{1}{2}\left(\frac{(n-1)^2}{4}-\alpha\right)<0,\quad \text{and} \\
\label{4}  & m(m-1)-\omega_1-2k(2m+1)R_0^2-2k(n-1)R_0^3\coth R_0>0;
\end{align}
here we have used our assumptions, $m(m-1)>\omega_1$ and $\alpha>\frac{(n-1)^2}{4}$. Then, for every $k\in(0,k_0]$ and every $r\geq R_0$, we can think of the left-hand side of \eqref{1} as a quadratic polynomial in $r$, having discriminant
\begin{align*}
\frac{\Delta_1}{4}&:=4k^2\left[\frac{(n-1)^2}{4}(\coth r)^2+2k(2m+1)-\alpha\right]\\
&<4k^2\left[\frac{1}{2}\left(\frac{(n-1)^2}{4}-\alpha\right)+2k(2m+1)\right]<0,
\end{align*}
by \eqref{2} and \eqref{3}. Thus \eqref{1} holds for every $k\in(0,k_0]$ and every $r\geq R_0$.

Now we claim that \eqref{5} holds for every $k\in(0,k_0]$ and every $r\geq0$, and thus \eqref{6} holds for every $k\in(0,k_0]$. Indeed, for every $r\geq R_0$ by \eqref{1} we have
\begin{align*}
  & m(m-1)+m(n-1)r\coth r-\frac{\omega_1}{(\sinh r)^2}r^2\\
  &\quad\quad+r^2[4k^2r^2-2(n-1)kr\coth r-(2k(2m+1)-\alpha)]\\
  &\geq m(n-1)r\coth r+m(m-1)-\omega_1>0.
\end{align*}
On the other hand for every $r\in[0,R_0]$ by \eqref{4} we have
\begin{align*}
  & m(m-1)+m(n-1)r\coth r-\frac{\omega_1}{(\sinh r)^2}r^2\\
  &\quad\quad+r^2[4k^2r^2-2(n-1)kr\coth r-(2k(2m+1)-\alpha)]\\
  &\geq m(m-1)-\omega_1-2k(2m+1)R_0^2-2k(n-1)R^3_0\coth R_0>0,
\end{align*}
and thus we conclude.
\end{proof}

\begin{proof}[Proof of Theorem \ref{thm1}]
  We consider the function defined for every $x\in\mathbb{H}^n$ by
  \begin{equation}\label{18}
  \Phi(x)\equiv\Phi(r(x))=C\Phi_0(r)=Cr^me^{-kr^2},
  \end{equation}
  with $C>0$ such that
  $$
  \int_0^\infty Cr^me^{-kr^2}(\sinh r)^{n-1}\,dr=1
  $$
  and with $m\geq2$, $m(m-1)>\omega_1$, $k\in(0,k_0]$ as in Lemma \ref{lemma1}. Note that $\Phi$ satisfies the assumptions of Proposition \ref{prop1}. Let $\alpha=\frac{\mu}{p-1}$, then in view of \eqref{13}, we have  $\alpha>\lambda_1(\mathbb{H}^n)=\frac{(n-1)^2}{4}$. Then by Lemma \ref{lemma1}
  \begin{equation*}
  \Delta\Phi+\alpha\Phi-\frac{\omega_1}{(\sinh r)^2}\Phi\geq0\qquad\text{ in }\mathbb{H}^n.
  \end{equation*}
  Setting for every $r>0$ and $t\in(0,T)$
  \begin{equation}\label{16}
  \tilde{u}(r,t)=e^{\alpha t}\int_\Omega \psi_1(\theta)u(r,\theta,t)\,d\sigma_\theta
  \end{equation}
  and for every $t\in(0,T)$
  \begin{equation}\label{17}
  G(t)=\int_0^\infty\tilde{u}(r,t)\Phi(r)(\sinh r)^{n-1}\,dr,
  \end{equation}
  from Proposition \ref{prop1} we have
  \begin{align}
\label{11} G'(t)&\geq e^{\left(\mu-(p-1)\alpha\right)t}(G(t))^p+\int_0^\infty\left(\Delta\Phi+\alpha\Phi-\frac{\omega_1}{(\sinh r)^2}\Phi\right)\tilde{u}(\sinh r)^{n-1}\,dr\\
\nonumber&\geq (G(t))^p
  \end{align}
  in $(0,T)$. Moreover
  \begin{align}
  \label{12} G(0)&=\int_0^\infty\tilde{u}(r,0)\Phi(r)(\sinh r)^{n-1}\,dr\\
  \nonumber&=\int_0^\infty\int_\Omega u_0(r,\theta)\psi_1(\theta)\Phi(r)(\sinh r)^{n-1}\,d\sigma_\theta dr>0,
  \end{align}
  since $u_0\geq0$, $u_0\not\equiv0$ in $\mathcal{C}$. In view of \eqref{11} and \eqref{12} we see by a standard qualitative analysis of the ODE that $G$ is defined on the maximal domain $(0,T)$, with $T<\infty$, and that
  $$
  G(t)\rightarrow\infty\qquad\text{ as }t\rightarrow T^-.
  $$
  Since $\Phi\in L^1(\mathbb{H}^n)$ we must have
  $$
  \|\tilde{u}(t)\|_{L^\infty((0,\infty))}\rightarrow\infty\qquad\text{ as }t\rightarrow T^-,
  $$
  and thus in turn
  $$
  \|u(t)\|_{L^\infty(\mathcal{C})}\rightarrow\infty\qquad\text{ as }t\rightarrow T^-.
  $$
  By direct integration of the differential inequality $G'\geq G^p$ one finally gets
  $$
  T\leq\frac{(G(0))^{1-p}}{p-1},
  $$
with $G(0)$ as in \eqref{12}.
\end{proof}

\begin{proof}[Proof of Theorem \ref{thm2}]
  Let $\alpha>\lambda_1(\mathbb{H}^n)$, $m\geq2$, $m(m-1)>\omega_1$, then $(p-1)\alpha-\mu>0$ and by Lemma \ref{lemma1} formula \eqref{6} holds for every $k>0$ sufficiently small. Arguing as in the proof of Theorem \ref{thm1}, we define $\Phi$, $\tilde{u}$ and $G$  as in formulas \eqref{18}, \eqref{16} and \eqref{17} respectively. By Proposition \ref{prop1} and Lemma \ref{lemma1} we deduce
  \begin{equation}\label{19}
  G'(t)\geq e^{-((p-1)\alpha-\mu)t}(G(t))^p
  \end{equation}
  for every $t\in(0,T)$, with
  \begin{align*}
  G(0)&=\int_0^\infty\int_\Omega u_0(r,\theta)\psi_1(\theta)\Phi(r)(\sinh r)^{n-1}\,d\sigma_\theta dr\\
   &=C\int_0^\infty\int_\Omega u_0(r,\theta)\psi_1(\theta)r^me^{-kr^2}(\sinh r)^{n-1}\,d\sigma_\theta dr>0,
  \end{align*}
  and
  $$
  C=\left(\int_0^\infty r^me^{-kr^2}(\sinh r)^{n-1}\,dr\right)^{-1}.
  $$
  Assume that \eqref{14} holds, then $G(0)>\left(\alpha-\frac{\mu}{p-1}\right)^{\frac{1}{p-1}}$. By direct integration of \eqref{19} we obtain
  $$
  G(t)\geq\left((G(0))^{1-p}+\frac{p-1}{(p-1)\alpha-\mu}(e^{-((p-1)\alpha-\mu)t}-1)\right)^{-\frac{1}{p-1}}
  $$
  for every $t\in(0,T)$. Then we must have
  \begin{equation}\label{20}
  0<T\leq T^*:=-\frac{1}{(p-1)\alpha-\mu}\log\left(1-\left(\alpha-\frac{\mu}{p-1}\right)(G(0))^{1-p}\right),
  \end{equation}
  with $T^*\in(0,\infty)$, and $$\lim_{t\rightarrow T^-}G(t)=+\infty,$$ and hence
  $$\lim_{t\rightarrow T^-}\|u(t)\|_{L^\infty(\mathcal{C})}=+\infty.$$
\end{proof}

\begin{proof}[Proof of Theorem \ref{thm2bis}] Arguing as in the proof of Theorem \ref{thm2}, but using Proposition \ref{prop1bis} instead of Proposition \ref{prop1}, we see that the function $G(t)$ defined in \eqref{eq1a} now satisfies
\begin{equation}\label{22}
  G'(t)\geq t^qe^{-(p-1)\alpha t}(G(t))^p
  \end{equation}
  for every $t\in(0,T)$, with $\alpha>\lambda_1(\mathbb{H}^n)$, and
  \begin{align*}
  G(0)&=C\int_0^\infty\int_\Omega u_0(r,\theta)\psi_1(\theta)r^me^{-kr^2}(\sinh r)^{n-1}\,d\sigma_\theta dr>0,
  \end{align*}
  with
  $$
  C=\left(\int_0^\infty r^me^{-kr^2}(\sinh r)^{n-1}\,dr\right)^{-1}.
  $$
  Now let $$H(t)=\int_0^ts^qe^{-(p-1)\alpha s}\,ds$$
  for $t\in(0,\infty]$ and observe that $H(t)$ is strictly increasing, with $H(\infty)\in\mathbb{R}$. By direct integration of \eqref{22} we see that we must have
  $$
  G(t)\geq\left((G(0))^{1-p}-(p-1)H(t)\right)^{-\frac{1}{p-1}}
  $$
  in $(0,T)$. If $$G(0)>((p-1)H(\infty))^{-\frac{1}{p-1}}$$
  then $G(t)$ must blow up in finite time, with
  $$
  0<T\leq T^*=H^{-1}\left(\frac{(G(0))^{1-p}}{p-1}\right),
  $$
  and hence $$\lim_{t\rightarrow T^-}\|u(t)\|_{L^\infty(\mathcal{C})}=+\infty.$$.
 \end{proof}

\section{Global existence: proofs}\label{exi}
\begin{lemma}\label{lemma2}
  If $p>1$ and $$\mu\leq(p-1)\lambda_1(\mathbb{H}^n)$$ then there exists a supersolution $v$ of
  $$v_t-\Delta v=e^{\mu t}v^p\qquad\text{ in }\mathbb{H}^n\times(0,\infty)$$ such that $v\in C^{2,1}_{x,t}(\mathbb{H}^n\times(0,\infty))\cap C^0(\mathbb{H}^n\times[0,\infty))\cap L^\infty(\mathbb{H}^n\times(0,\infty))$ and $v>0$ in $\mathbb{H}^n\times[0,\infty)$.
\end{lemma}
\begin{proof}
  The construction of the supersolution $v$ in case $\mu<(p-1)\lambda_1(\mathbb{H}^n)$ can be found in Subsection 5.2 in  \cite{BaPoTe}. The case $\mu=(p-1)\lambda_1(\mathbb{H}^n)>0$ with $\mu>\frac{2}{3}\lambda_1(\mathbb{H}^n)$ is considered in Subsection 5.3 in \cite{BaPoTe}. Finally, the case $\mu=(p-1)\lambda_1(\mathbb{H}^n)>0$ with $0<\mu\leq\frac{2}{3}\lambda_1(\mathbb{H}^n)$ is detailed in the proof of Theorem 1 in \cite{WaYi}.
\end{proof}

\begin{lemma}\label{lemma3}
  Let $\Omega\subseteq\mathbb{S}^{n-1}$ and $\mathcal{C}=(0,\infty)\times\Omega$, then for every $x_0\in\partial\mathcal{C}$ there exists a local barrier function at $x_0$ for the operator $-\Delta$. That is, if we define $\mathcal{U}:=B_\delta(x_0)\cap\mathcal{C}$, there exist $\delta>0$ and $h\in C^2(\mathcal{U})\cap C^0(\overline{\mathcal{U}})$ such that $h>0$ in $\overline{\mathcal{U}}\setminus\{x_0\}$, $h(x_0)=0$ and $\Delta h\leq-1$ in $\mathcal{U}$.
\end{lemma}
\begin{proof}
  The proof of this result can be obtained arguing as in Theorem 2 and in the remarks following Theorem 3 in \cite{Mil}, since for every $x_0\in\mathcal{C}$ the exterior cone condition at $x_0$ holds.
\end{proof}

\begin{lemma}\label{lemma4}
Let $\Omega\subseteq\mathbb{S}^{n-1}$ and $\mathcal{C}=(0,\infty)\times\Omega$, then for every $x_0\in\partial\mathcal{C}$, $t_0\geq0$ there exists a local barrier function at $(x_0,t_0)$ for the operator $\partial_t-\Delta$. That is, if we define $\mathcal{V}:=[B_\delta(x_0)\times(t_0-\delta,t_0+\delta)]\cap[\mathcal{C}\times[0,\infty)]$, there exist $\delta>0$ and $\zeta\in C^2(\mathcal{V})\cap C^0(\overline{\mathcal{V}})$ such that $\zeta>0$ in $\overline{\mathcal{V}}\setminus\{(x_0,t_0)\}$, $\zeta(x_0,t_0)=0$ and $\zeta_t-\Delta \zeta\geq1$ in $\mathcal{V}$.
\end{lemma}

\begin{proof}
  Let $h$ be a local barrier function at $x_0$ for the operator $-\Delta$, as given by Lemma \ref{lemma3}. Then the function $$\zeta(x,t)=Ch(x)+(t-t_0)^2$$ satisfies all the desired properties, for $C>0$ large enough.
\end{proof}

\begin{proof}[Proof of Theorem \ref{thm3}]
  For every $j\in\mathbb{N}$ let $D_j\subset\mathcal{C}$ be a bounded open domain with boundary of class $C^1$, such that $\overline{D}_j\subset D_{j+1}$ for every $j\in\mathbb{N}$ and $\cup_{j\in\mathbb{N}}D_j=\mathcal{C}$. For every $j\in\mathbb{N}$ let $\eta_j\in C^{\infty}_0(D_j)$ be such that $0\leq\eta_j\leq1$ in $D_j$ and $\eta_j=1$ in $D_{j-1}$.

 Let $v$ be the supersolution of \eqref{15} given by Lemma \ref{lemma2} and assume
 $$0\leq u_0(x)\leq v(x,0)\qquad\text{ for every }x\in\mathcal{C}.$$
 By classical results on uniformly parabolic operators (see e.g. \cite[Cap. 7]{Fri}, \cite[Cap. 14]{LaSoUr}) for every $j\in\mathbb{N}$ there exists a unique classical solution $u_j$ of
  \begin{equation*}
    \begin{cases}
      \partial_tu_j-\Delta u_j=e^{\mu t}u_j^p, & \mbox{in } D_j\times(0,T), \\
      u_j=0, & \mbox{in } \partial D_j\times(0,T), \\
      u_j=\eta_ju_0, & \mbox{on }D_j\times\{0\},
    \end{cases}
  \end{equation*}
  for some maximal existence time $T>0$, which may depend on $j$. Since $$0\leq\eta_ju_0\leq u_0\leq v(\cdot,0)\qquad\text{ in }D_j,$$ by standard parabolic comparison principles for every $j\in\mathbb{N}$ we have
  \begin{equation}\label{21}
    0\leq u_j\leq v\qquad\text{ in }D_j\times(0,T).
  \end{equation}
  Since $v\in L^\infty(\mathbb{H}^n\times(0,\infty))$ we must have $T=\infty$. By standard a priori estimates (see e.g. \cite[Cap. 7]{Fri}, \cite[Cap. 14]{LaSoUr}), up to taking subsequences, we have that for every $K\subset\mathcal{C}$ compact and every $\tau>\varepsilon>0$ the sequence $u_j$ converges in $C^{2,1}_{x,t}(K\times[\varepsilon,\tau])$ to a function $u\in C^{2,1}_{x,t}(\mathcal{C}\times(0,\infty))$, which satisfies
  $$
  u_t-\Delta u=e^{\mu t}u^p\qquad\text{ in }\mathcal{C}\times(0,\infty).
  $$
  By \eqref{21} we have $$0\leq u\leq v \quad \text{ in } \, \mathcal{C}\times(0,\infty),$$ and hence $u\in L^\infty(\mathcal{C}\times(0,\infty))$. Moreover by Lemma \ref{lemma4}, by a standard parabolic local barrier argument, we see that $u$ attains continuously the initial and the boundary data
  \begin{align*}
    u(x,0)&=u_0(x) \qquad \textrm{ in }\mathcal{C}, \\
    u(x,t)&=0 \qquad\qquad \textrm{ in }\partial\mathcal{C}\times(0,\infty).
  \end{align*}
  Thus $u$ is a bounded classical solution of problem \eqref{7}.
\end{proof}

\begin{proof}[Proof of Theorem \ref{thm3bis}] The proof follows along the same lines as that of Theorem \ref{thm3}. Here the only difference is that, instead of using Lemma \ref{lemma2}, one has to consider a supersolution $v$ of
  $$v_t-\Delta v=t^qv^p\qquad\text{ in }\mathbb{H}^n\times(0,\infty)$$ such that $v\in C^{2,1}_{x,t}(\mathbb{H}^n\times(0,\infty))\cap C^0(\mathbb{H}^n\times[0,\infty))\cap L^\infty(\mathbb{H}^n\times(0,\infty))$ and $v>0$ in $\mathbb{H}^n\times[0,\infty)$. The construction of a supersolution with the desired properties can be found in Subsection 5.2 in  \cite{BaPoTe}.
\end{proof}

\

\

\

\textbf{Acknowledgements.} The authors are members of the {\em GNAMPA, Gruppo Nazionale per l'Analisi Matematica, la Probabilità e le loro Applicazioni} of INdAM.

\bibliographystyle{plain}

\begin{thebibliography}{99}
\bibitem{BaPoTe} C. Bandle, M.A. Pozio, A. Tesei, {\it The Fujita exponent for the Cauchy problem in the hyperbolic space}, J. Differential Equations {\bf 251} (2011), no. 8, 2143–2163.
\bibitem{BaLe} C. Bandle, H. Levine, {\it On the existence and nonexistence of global solutions of reaction diffusion equations in sectorial domains} {\bf 316} (1989), 595--622\,.

\bibitem{Fri} A. Friedman, Partial Differential Equations of Parabolic Type, Dover Publications, New York (1992).

\bibitem{Fuj}  H. Fujita, {\it On the blowing up of solutions of the Cauchy problem for $u_t = \Delta u + u^{1+\alpha}$}, J. Fac. Sci.
Univ. Tokyo Sect. I {\bf 13} (1966), 109--124.

\bibitem{Hay}  K. Hayakawa, {\it On nonexistence of global solutions of some semilinear parabolic differential equations}, Proc. Japan Acad. {\bf 49} (1973), 503--505.

\bibitem{Kapl} S. Kaplan, {\it On the growth of quasilinear parabolic equations}, Comm. Pure Appl. Math. {\bf 16} (1963), 305--330.

\bibitem{LaSoUr} O.A. Ladyzhenskaya, V.A. Solonnikov, N.A. Uraltseva, Linear and Quasilinear Equations of
Parabolic Type, Nauka, Moscow (1967) (English translation: series Transl. Math. Monographs, 23
AMS, Providence, RI, 1968)

\bibitem{Lun} A. Lunardi, {\it Analytic Semigroups and Optimal Regularity in Parabolic Problems},
Birkhauser/Springer Basel AG, Basel, 1995.

\bibitem{Mil} K. Miller, {\it Barriers on cones for uniformly elliptic operators}, Ann. Mat. Pura Appl. {\bf 76} (1967), no. 4, 93–105.

\bibitem{PuTe}  F. Punzo, A. Tesei, {\it  Uniqueness of  solutions to degenerate elliptic problems
 with unbounded coefficients}, Ann. Inst. H. Poincar\'{e}, An. Non Lin. {\bf 26}  (2009), 2001-2024.

\bibitem{WaYi} Z. Wang, J. Yin, {\it A note on semilinear heat equation in hyperbolic space}, J. Differential Equations {\bf 256} (2014), 1151–1156.


\end{thebibliography}

\end{document}